\documentclass[10pt]{article}

\usepackage{amsthm}
\usepackage{amsmath}
\usepackage{amssymb}
\usepackage{mathrsfs}

\def\A{{\mathcal A}}

\def\Th{\Theta}

\def\N{\mathbb{N}}

\def\X{\mathscr X}

\def\AA{\mathfrak A}

\def\R{\mathcal{R}}

\def\S{\mathcal S}

\def\Ai{\mathcal A^\infty}
\def\AAi{\mathfrak A^\infty}

\def\CC{{\mathfrak C}}

\def\RR{{\mathfrak R}}

\def\I{{\rm 1\kern-.26em I}}

\def\si{\sigma}
\def\Si{\Sigma}

\def\1{\mathfrak{1}}
\def\0{\mathfrak{0}}

\def\<{\langle}
\def\>{\rangle}

\providecommand{\CC}{\mathfrak{C}}

\def\A{{\mathcal A}}
\def\AA{{\mathfrak A}}
\def\Ai{{\mathcal A^\infty}}
\def\AAi{{\mathfrak A^\infty}}

\def\Th{\Theta}

\def\N{\mathbb{N}}

\def\X{\mathscr X}

\def\[{\left[}
\def\]{\right]}
\def\<{\left<}
\def\>{\right>}
\def\({\left(}
\def\){\right)}

\def\R{\mathcal{R}}

\def\S{\mathscr S}

\def\CC{{\mathfrak C}}

\def\RR{{\mathfrak R}}

\def\I{{\rm 1\kern-.26em I}}

\def\si{\sigma}
\def\Si{\Sigma}

\def\1{\mathfrak{1}}
\def\0{\mathfrak{0}}

\def\<{\langle}
\def\>{\rangle}

\newcommand{\bb}{\mathbb}

\newcommand{\scalar}[2]{\langle #1, #2 \rangle}
\newcommand{\norm}[1]{\Vert #1 \Vert }
\newcommand{\normrum}[2]{{\norm {#1}}_{#2}}
\newcommand{\opn}{\operatorname}

\newcommand{\Cal}{\mathscr}
\newcommand{\gk}{\mathfrak}

\providecommand{\CC}{\mathfrak{C}}

\providecommand{\ie}{i.~e.}

\newtheorem{Theorem}{Theorem}[section]
\newtheorem{Remark}[Theorem]{Remark}
\newtheorem{Lemma}[Theorem]{Lemma}

\newtheorem{Corollary}[Theorem]{Corollary}
\newtheorem{Proposition}[Theorem]{Proposition}
\newtheorem{Definition}[Theorem]{Definition}

\numberwithin{equation}{section}

\begin{document}

\title{Rieffel deformation and twisted crossed products}

\author{I. Belti\c t\u a$\,^{*}$ and  M. M\u antoiu$\,^{**}$}

\maketitle

\footnote*{

\begin{itemize}
\item[$^*$] Institute
of Mathematics Simion Stoilow of the Romanian Academy, P.O.  Box
1-764, Bucharest, \\ RO-70700, Romania,
Email: {\tt ingrid.beltita@imar.ro}
\item[$^**$] Departamento de Matem\'aticas, Universidad de Chile, Las Palmeras 3425, Casilla 653,
Santiago, Chile, \\
Email: {\tt mantoiu@uchile.cl}
\end{itemize}}

\date{\small}
\maketitle

\begin{abstract}
To a continuous action of a vector group on a $C^*$-algebra,
twisted by the imaginary exponential of a symplectic form, one
associates a Rieffel deformed algebra as well as a twisted crossed
product. We show that the second one is isomorphic to the tensor
product of the first one with the $C^*$-algebra of compact
operators in a separable Hilbert space and we indicate some
applications.
\end{abstract}

\maketitle

\section{Introduction}\label{duci}

In order to provide a unified framework for a large class of examples in deformation quantization, Marc Rieffel \cite{Rie1}
significantly extended the basic part of the Weyl pseudodifferential calculus. Rieffel's calculus starts from the
action $\Th$ of a finite-dimensional vector space $\Xi$ on a $C^*$-algebra $\A$, together with a skew-symmetric  linear
operator $J:\Xi\rightarrow\Xi$ that serves to twist the product on $\A$. Using $J$ one defines first a new composition law $\#$ on
the set of smooth elements of $\A$ under the action and then a completion is taken in a suitable $C^*$-norm. The outcome is a new
$C^*$-algebra $\AA$, also endowed with an action of the vector space $\Xi$. The corresponding subspaces of smooth vectors under
the two actions, $\Ai$ and $\AAi$, respectively, coincide. In \cite{Rie1} the functorial properties of the correspondence
$\A\mapsto\AA$ are studied in detail and many examples are given. It is also shown that one gets a strict deformation quantization
of a natural Poisson structure defined on $\Ai$ by the couple $(\Th,J)$.

Assuming $J$ non-degenerate (so it defines a symplectic form on $\Xi$), one gets a twisted action
$(\Th,\kappa)$ of $\Xi$ on the $C^*$-algebra $\A$, where
$\kappa$ is the $2$-cocycle on $\Xi$ given by $(X,Y)\mapsto \kappa(X,Y):=\exp(iX\cdot JY)$. To such a twisted action one
associates canonically \cite{PR1,PR2} a twisted crossed product $C^*$-algebra $\A\rtimes_\Th^\kappa\Xi$, whose representations are
determined by covariant representations of the quadruplet $(\A,\Th,\kappa,\Xi)$.

In the present article we are going to show that the two $C^*$-algebras $\AA$ and $\A\rtimes_\Th^\kappa\Xi$ that can be
constructed from the data $(\A,\Th,\kappa,\Xi)$ are actually stably isomorphic. This happens in a particularly precise way: one has
an isomorphism (called the canonical mapping) $M:\mathscr K\otimes\AA\rightarrow\A\rtimes_\Th^\kappa\Xi$,
where $\mathscr K$ is an elementary $C^*$-algebra, i.e. it is faithfully represented
as the ideal of all compact operators in a separable Hilbert space. The mapping $M$
is naturally defined first between convenient Fr\'echet subalgebras
(vector-valued Schwartz spaces); the extension to a $C^*$-isomorphism needs a non-trivial isometry argument.

Such stable isomorphism has standard consequences \cite{RW}:
the (closed, bi-sided self-adjoint) ideals of the two algebras $\AA$ and $\A\rtimes_\Th^\kappa\Xi$ are in
one-to-one correspondence, the spaces of primitive ideals are homeomorphic and the two representation theories are identical. By
using basic information about the twisted crossed product, we also get a simple proof of the known fact
\cite{Rie3,Kp} that the $K$-groups of the Rieffel deformed algebra $\AA$ are the same as those of the initial algebra $\A$.
A covariant morphism $\mathcal R:(\A^1,\Th^1)\rightarrow(\A^2,\Th^2)$ can be raised both to a morphism
$\mathfrak R:\AA^1\rightarrow\AA^2$ and to a morphism
$\mathcal R^\rtimes:\A^1\rtimes_{\Th^1}^\kappa\Xi\rightarrow\A^2\rtimes_{\Th^2}^\kappa\Xi$.
The canonical mappings $M^1,M^2$ have the intertweening property $\mathcal R^\rtimes\circ M^1=M^2\circ({\rm id}\otimes\mathfrak R)$.

When the initial algebra $\A$
is commutative, it is associated by Gelfand theory with a locally compact topological dynamical system $(\Sigma,\Th,\Xi)$.
Under some assumptions on this system, one can get information on the primitive ideal space of the $C^*$-algebra $\AA$.
A choice of an invariant measure on $\Si$ leads to $L^2$-orthogonality relations for the canonical mapping $M$.
We hope to continue to investigate the canonical mappings in the commutative case
(the one closest in spirit with traditional pseudodifferential theory), having in view a more detailed study of representations,
modulation spaces and applications to spectral analysis \cite{Ma}.

\section{Involutive algebras associated to a twisted $C^*$-dynamical system}\label{sectra}

We shall recall briefly, in a slightly particular setting, some constructions and results concerning twisted crossed products
algebras and Rieffel's pseudodifferential calculus.

The common starting point is a $2n$-dimensional real vector space $\Xi$ endowed with a symplectic form $[\![\cdot,\cdot]\!]$. When
needed we are going to suppose that $\Xi=\X\times\X^*$, with $\X^*$ the dual of the $n$-dimensional vector space $\X$, and that for
$X:=(x,\xi),\,Y:=(y,\eta)\in\Xi$, the symplectic form reads $[\![X,Y]\!]:=y\cdot\xi-x\cdot \eta$.

An action $\Th$ of $\Xi$ by automorphisms of a (maybe non-commutative) $C^*$-algebra $\A$ is also given. For
$(f,X)\in\A\times\Xi\,$ we are going to use the notations $\Theta(f,X)=\Th_X(f)=\Th_f(X)\in\A$ for the $X$-transform of
the element $f$. This action is assumed strongly continuous, i.e. for any $f\in\A$ the mapping $\Xi\ni X\mapsto\Th_X(f)\in\A$ is
continuous. The initial object, containing {\it the classical data}, is a quadruplet
$\left(\A,\Th,\Xi,[\![\cdot,\cdot,]\!]\right)$ with the properties defined above.

To arrive at twisted crossed products, we define
\begin{equation}\label{caf}
\kappa:\Xi\times\Xi\rightarrow\mathbb T:=\{\lambda\in\mathbb C\mid
|\lambda|=1\}\,, \ \ \ \ \
\kappa(X,Y):=\exp\left(-\frac{i}{2}\,[\![X,Y]\!]\right)
\end{equation}
and notice that it is a group $2$-cocycle, i.e. for all
$X,Y,Z\in\Xi$ one has
$$
\kappa(X,Y)\,\kappa(X+Y,Z)=\kappa(Y,Z)\,\kappa(X,Y+Z)\,,\ \ \ \ \
\kappa(X,0)=1=\kappa(0,X)\,.
$$
Thus the classical data is converted into $\(\A,\Th,\Xi,\kappa\)$, a very particular case of {\it  twisted $C^*$-dynamical system}
\cite{PR1,PR2}.
To any twisted $C^*$-dynamical system one associates canonically a $C^*$-algebra $\A\rtimes_\Th^\kappa\Xi$ (called {\it twisted
crossed product}). This is the enveloping $C^*$-algebra of the Banach $^*$-algebra
$\(L^1(\Xi;\A),\diamond,^\diamond,\parallel\cdot\parallel_1\)$,
where
\begin{equation*}\label{aca}
\parallel G\parallel_1:=\int_\Xi dX\parallel G(X)\parallel_\A,\ \ \ \ \ G^\diamond(X):=G\(-X\)^*
\end{equation*}
and (symmetrized version of the standard form, cf. Remark \ref{alta})
\begin{equation}\label{ucu}
(G_1\diamond G_2)(X):=\int_\Xi
dY\,\kappa(X,Y)\,\Th_{(Y-X)/2}\[G_1(Y)\]\,\Th_{Y/2}\left[G_2(X-Y)\right].
\end{equation}
In \cite{PR1,PR2} $\A$ is supposed separable; since our cocycle is explicit and very simple, this will not be needed here.

We turn now to {\it Rieffel deformation} \cite{Rie1,Rie2}. Let us denote by $\Ai$ the family of elements $f$ such that the mapping
$\,\Xi\ni X\mapsto\Th_X(f)\in\mathcal A\,$ is $C^\infty$. It is a dense $^*$-algebra of $\A$ and also a Fr\'echet algebra with the
family of semi-norms
\begin{equation}\label{semicar}
|f|_\A^k:=\sum_{|\alpha|=
k}\frac{1}{\alpha!}\parallel\partial_X^\alpha\left[\Th_X(f)\right]_{X=0}\parallel_\A\
\equiv\sum_{|\alpha|= k}\frac{1}{\alpha!}\parallel\delta^\alpha (f)\parallel_\A \
, \ \ \ \ k\in\N\, .
\end{equation}
To quantize the above structure, one keeps the involution but introduce on $\Ai$ the product
\begin{equation}\label{rodact}
f\,\#\,g:=2^{2n}\int_\Xi\int_\Xi
dYdZ\,e^{2i[\![Y,Z]\!]}\,\Th_Y(f)\,\Th_Z(g)\ ,
\end{equation}
suitably defined by oscillatory integral techniques. Thus one gets a $^*$-algebra $(\Ai,\#\,,^*)$, which admits a $C^*$-completion
$\AA$ in a $C^*$-norm $\parallel\cdot\parallel_\AA$ defined by Hilbert module techniques; we are going to call $\AA$ {\it the
R-deformation of $\A$}. The action $\Th$ leaves $\Ai$ invariant and extends to a strongly continuous action on the $C^*$-algebra
$\AA$, that will also be denoted by $\Th$. The space $\AAi$ of $C^\infty$-vectors coincide with $\Ai$, even topologically, i.e.
the family (\ref{semicar}) on $\Ai=\AAi$ is equivalent to the family of semi-norms
\begin{equation}\label{semon}
|f|_\AA^k:=\sum_{|\alpha|=
k}\frac{1}{\alpha!}\parallel\partial_X^\alpha\left[\Th_X(f)\right]_{X=0}\parallel_\AA\
\equiv\sum_{|\alpha|= k}\frac{1}{\alpha!}\parallel\delta^\alpha (f)\parallel_\AA
\ ,\ \ \ \ \ k\in\N\ .
\end{equation}

An important particular case is obtained when $\A$ is the $C^*$-algebra $BC_{{\rm u}}(\Xi)$ of bounded uniformly continuous
functions on the group $\Xi$, which is invariant under translations, i.e. if $a\in\A$ and $X\in\Xi$, then $\left[\mathcal
T_X(a)\right](\cdot):=a(\cdot-X)\in\A$. Notice that the $^*$-algebra of smooth vectors coincides with $BC^\infty(\Xi)$,
the space of all smooth complex functions on $\Xi$ with bounded derivatives of every order. In this case Rieffel's construction,
done for $\Th=\mathcal T$, reproduces essentially the standard Weyl calculus; we are going to use the special notations $\sharp$
(instead of $\#$) for the corresponding composition law and $\mathscr B(\Xi)$ for the R-deformation of $BC_{{\rm u}}(\Xi)$.

One can also consider $C^*$-subalgebras $\A$ of $BC_{{\rm u}}(\Xi)$ that are invariant under translations. An important one
is $C_0(\Xi)$, formed of all the complex continuous functions on $\Xi$ that decay at infinity. Its Rieffel deformation will be
denoted by $\mathscr K (\Xi)$; it contains the Schwartz space $\S(\Xi)$ densely. By Example 10.1 and Proposition 5.2 in
\cite{Rie1} it is elementary, i.e isomorphic to the $C^*$-algebra of all compact operators in a separable Hilbert space.

Following \cite{Rie1}, we introduce the Fr\'echet space $\S(\Xi;\AAi)$ composed of smooth functions
$F:\Xi\rightarrow\Ai=\AAi$ with derivatives that decay rapidly with respect to all $|\cdot|_\AA^k$\,. The relevant seminorms on
the space $\S(\Xi;\AAi)$ are $\{\parallel\cdot\parallel^{k,\beta,N}_\AA\,\mid\,k,N\in\N,\ \beta\in\N^{2n}\}$ where
\begin{equation}\label{cic}
\parallel F\parallel^{k,\beta,N}_\AA :=\underset{{X\in\, \Xi}}{\sup}
\{(1+|X|)^N |(\partial^\beta F)(X)|^k_\AA\},
\end{equation}
and the index $\AA$ can be replaced by $\A$, by the argument above. We are going to use repeatedly the identification of
$\S\left(\Xi;\AAi\right)$ with the topological tensor product $\S(\Xi)\hat\otimes\,\AAi$ (recall that the Fr\'echet space
$\S(\Xi)$ is nuclear). On it (and on many other larger spaces) one can define obvious actions $\mathfrak T:=\mathcal T\otimes 1$ and
$\mathcal T\otimes\Th$ of the vector spaces $\Xi$ and $\Xi\times\Xi$, respectively. Explicitly, for all $A,Y,X\in\Xi$,
one sets $\[\mathfrak T_A(F)\](X):=F(X-A)$ and $\[\(\mathcal
T_A\otimes\Th_Y\)F\](X):=\Th_Y\[F(X-A)\]$. Then on $\S\left(\Xi;\AAi\right)$ one can introduce the composition law
\begin{equation}\label{argrur}
\(F_1\square F_2\)(X)=2^{2n}\int_\Xi\int_\Xi
dAdB\,e^{-2i[\![A,B]\!]}\,\[\mathfrak T_A(F_1)\](X)\#\,
\[\mathfrak T_B(F_2)\](X)=
\end{equation}
\begin{equation}\label{palma}
=2^{4n}\int_\Xi\int_\Xi\int_\Xi\int_\Xi dA\,dB\,dYdZ\,e^{-2i[\![A,B]\!]}\,e^{2i[\![Y,Z]\!]}
\end{equation}
$$
\quad\quad\[\left(\mathcal T_A\otimes\Th_Y\right)(F_1)\](X)\,\[\left(\mathcal T_B\otimes\Th_Z\right)(F_2)\](X)\,.
$$

Notice that the last expression should be interpreted as an oscillatory integral \cite{Rie1} and that it involves the
multiplication in the $C^*$-algebra $\A$. If the involution is given by $F^\square(X):=F(X)^*,\ \forall\,X\in \Xi$, it can be
shown that one gets a Fr\'echet $^*$-algebra.

\begin{Remark}\label{ciuhat}
{\rm We recall that $\Ai=\AAi$, even topologically, but the algebraic structures are different. When the forthcoming arguments will
involve the composition $\#$\,, in order to be more suggestive, we will use the notation $\S\left(\Xi;\AAi\right)$. In other
situations the notation $\S\left(\Xi;\Ai\right)$ will be more natural. For instance, it is easy to check that
$\S\left(\Xi;\Ai\right)$ is a (dense) $^*$-subalgebra of the Banach $^*$-algebra
$\(L^1(\Xi;\A),\diamond,^\diamond,\parallel\cdot\parallel_1\)$, which is defined in terms of the product $\cdot$ on $\A$ and has a
priori nothing to do with the composition law $\#$\,. Proposition \ref{bezu} is a good illustration for this distinction.}
\end{Remark}

\begin{Remark}\label{uhu}
{\rm One can also consider $BC_{\rm u}\left(\Xi;\AA\right)$, the $C^*$-algebra of all bounded and uniformly continuous
functions $F:\Xi\rightarrow\AA$. Rieffel deformation can also be applied to the new classical data
$\left(BC_{\rm u}\left(\Xi;\AA\right), \mathfrak T,\Xi,-[\![\cdot,
\cdot]\!]\right)$, getting essentially (\ref{argrur}) as the corresponding composition law.
By using the second part (\ref{palma}) of the formula, this can also be regarded as the Rieffel composition constructed from
the extended twisted $C^*$-dynamical system $\left(BC_{\rm u}\left(\Xi;\A\right),\,\mathcal
T\otimes\Th,\,\Xi\times\Xi,\,\overline\kappa\otimes\kappa\right)\,$.
This will not be needed in this form. But we are going to use below the fact that for elements $f,g\in\AAi$, $a,b\in\S(\Xi)$ one
has $(a\otimes f)\square\,(b\otimes g)=(b\,\sharp\,a)\otimes(f\#g)\,$, so $\square$ can be seen as the tensor product between $\#$
and the law opposite to $\sharp$. By Proposition 2.1 in \cite{Rie2}, one can identify $\mathscr K(\Xi)\otimes\AA$ with the
R-deformation of $C_0(\Xi)\otimes\A\equiv C_0(\Xi;\A)$ and $\mathscr B(\Xi)\otimes\AA$ with the R-deformation of $BC_{{\rm
u}}(\Xi)\otimes\A$\,.}
\end{Remark}

\section{The Schr\" odinger representation}\label{grid}

We are going to denote by $\mathbb B(\mathcal M,\mathcal N)$ the space of all linear continuous
operators acting between the topological vector spaces $\mathcal M$ and $\mathcal N$ and use the abbreviation $\mathbb B(\mathcal
M)$ for $\mathbb B(\mathcal M,\mathcal M)$.

Let us recall that ${\Cal X}$ is a finite-dimensional vector space. The corresponding \emph{Heisenberg algebra} ${\mathfrak h}_{{\Cal
X}}={\Cal X}\times{\Cal X}^*\times{\mathbb R}$ is the Lie algebra with the bracket
$$
[(x,\xi,t),(y, \eta,s)]:=(0,0,y \cdot \xi-x \cdot \eta)\,.
$$
We use notations as $\bar X=(x, \xi, t)$ and $X= (x, \xi)$. The \emph{Heisenberg group} ${\mathbb H}_{{\Cal X}}$ is just
${\mathfrak h}_{{\Cal X}}$ thought of as a group with the multiplication~$\ast$ defined by
$$
{\bar X}\ast {\bar Y}={\bar X}+{\bar Y}+\frac{1}{2}[{\bar X},{\bar Y}], \quad {\bar X},{\bar Y}\in {\mathbb H}_{{\Cal X}}.
$$
The unit element is $0\in{\mathbb H}_{{\Cal X}}$ and the inversion mapping given by ${\bar X}^{-1}:=-{\bar X}$.

\textit{The Schr\"odinger representation} is the unitary representation $\Pi\colon{\mathbb H}_{{\Cal
X}}\to{\mathbb B}(\mathcal L)$ in the Hilbert space $\mathcal L:=L^2({\Cal X})$, defined by
\begin{equation}\label{sch}
\left[\Pi (\bar X) u\right](y)=[\Pi(x, \xi,t)u](y)={\rm e}^{{\rm
i}(y \cdot \xi+\frac{1}{2}x \cdot \xi+t)}u(y+x) \quad\text{ for
a.e. $y \in {\Cal X}$}
\end{equation}
for arbitrary $u\in L^2({\Cal X})$ and $\bar X=(x,\xi,t)\in{\mathbb H}_{{\Cal X}}$. When restricted to $\Xi=
{\Cal X}\times {\Cal X}^*$ (which is not a subgroup and should be regarded as a quotient of ${\mathbb H}_{{\Cal X}}$),
$\Pi$ becomes a projective representation that will be denoted by $\pi$: it satisfies
$$
\pi(X)\,\pi(Y)=\kappa(X,Y)\,\pi(X+Y),\ \ \forall X,Y\in\Xi\,.
$$

{\it The Wigner distributions  defined by} $\pi$ are given by
$$
{\Cal W}(u,v) := \mathcal F(\scalar{u}{\pi(\cdot) v}, \quad u,v\in \mathcal L.
$$
We used {\it the symplectic Fourier transform}
$$
(\mathcal F a)(X):=\int_\Xi dY e^{-i[\![X,Y]\!]}a(Y)
$$
and forced it to be $L^2$-unitary and satisfy $\mathcal F^2={\rm id}$, by a suitable choice of Lebesgue measure $dY$ on $\Xi$.
Recall that ${\Cal W}(u,v)\in {\Cal S}(\Xi)$ when $u$, $v \in {\Cal S}(\Cal X)$, ${\Cal W}\colon \mathcal L\times \mathcal L \to
L^2(\Xi)$ is an isometry and extends to a unitary mapping ${\Cal W}\colon L^2(\Cal X)\otimes \overline{L^2(\Cal X)}\to L^2(\Xi)$.
{\it The Weyl pseudodifferential calculus} is then a linear isomorphism
 \begin{equation}\label{mirtan}
\opn{Op} \colon {\Cal S}'(\Xi)\to \mathbb B[{\Cal S}(\Cal X), {\Cal S}'(\Cal X)],\ \quad
\scalar{v}{\opn{Op}(a)u} = \scalar{\overline{ {\Cal W}(v,u)}}{a}.
\end{equation}
Recall  also that $\opn{Op} [{\Cal W}(u,v)]= \<\cdot\mid v\> u$ for all $u,v\in \mathcal L$, and $\opn{Op}\colon L^2(\Xi)\to {\gk
S}_2(\mathcal L)$ (Hilbert-Schmidt operators) is unitary. For $a$, $b \in {\Cal S}'(\Xi)$, $a\,\sharp\,b$ is the symbol of the
operator $\opn{Op}(a)\opn{Op}(b)$ whenever this is well-defined and continuous from ${\Cal S}(\Cal X)$ to ${\Cal S}'(\Cal X)$. Of
course, the symbol $\sharp$ is an extension of the one used in the previous section. The action of $\opn{Op}(a)$ on $\S(\X)$ or
$\mathcal L:=L^2(\X)$ (under various assumptions on the symbol $a$ and with various interpretations) is given by
\begin{equation}\label{eil}
\left[\opn{Op}(a)v\right](x):=\int_\X
dy\int_{\X^*}d\xi\,e^{i(x-y)\cdot\xi}\,a\left(\frac{x+y}{2},
\xi\right)v(y).
\end{equation}

Consider next the space of operators
$$
{\bb B}_{\text{u}} (\mathcal L) = \{ T\in {\bb B} (\mathcal L) \mid \Xi \ni X\to \pi( X)  T\pi (-X) \in {\bb B} (\mathcal L) \;
\text{is norm continuous}\}.
$$
Then ${\bb B}_{\text{u}} (\mathcal L)$ is a proper $C^*$-subalgebra of ${\bb B} (\mathcal L)$ with the norm given by
the operator norm and involution given by Hilbert space adjoint, and it contains the ideal ${\bb K}(\mathcal L)$
of compact operators on $\mathcal L$ (see \cite[Thm. 1.1]{Cor79}). The representation
$$
\pi \otimes {\bar \pi}\colon \Xi \to {\bb B}\left[\mathbb
B_\text{u}(\mathcal L)\right], \quad (\pi \otimes {\bar \pi}) (X)
T= \pi (X) T\pi(-X)
$$
is then strongly continuous. Let ${\bb B}_{\text{u}}^\infty(\mathcal L)$ be the  space of smooth vectors
for this representation. Then   ${\bb B}_{\text{u}}^\infty(\mathcal L)$ is dense in
 ${\bb B}_{\text{u}}(\mathcal L)$ \cite[Thm.1.1]{Cor79}, and consists precisely of those Weyl pseudo-differential operators
 with symbols in $BC^\infty(\Xi)$ (\cite[Thm.1.2]{Cor79} and \cite[Thm. 2.3.7]{Fo}).

\begin{Lemma}\label{precisely}
The Weyl calculus $\opn{Op}$ realizes an isomorphism between $\Cal B(\Xi)$ (the R-deformation of $BC_{{\rm u}}(\Xi)$) and $\mathbb
B_{{\rm u}}(\mathcal L)$. The image through $\opn{Op}$ of $\Cal K(\Xi)$ is precisely $\mathbb K(\mathcal L)$.
\end{Lemma}

\begin{proof}
Indeed, recall that when $a\in BC^\infty(\Xi)$, the norm $\parallel a\parallel_{\Cal B(\Xi)}$ of $a$ in the Rieffel algebra
is given by the norm of the operator $L_a\colon {\S}(\Xi) \to {\S}(\Xi)$, $L_a(b) = a\# b$, where on ${\S}(\Xi)$ one considers
the $L^2$-norm (a particular case of \cite[Prop. 4.15]{Rie1}). Taking $b= \Cal
{W}(u,v)$, with $u,v \in {\S}(\Cal X)$, one gets that $a\#  \Cal {W}(u,v)= {\Cal W}(\opn{Op}(a)u,v)$, hence
$$
\norm{L_a[\Cal {W}(u,v)]}_{L^2(\Xi)}= \norm{v}\,
\norm{\opn{Op}(a)u}, \quad u,v\in {\S}(\Cal X).
$$
Thus $\norm{\opn{Op}(a)}_{\mathbb B(\mathcal L)}\le \norm{a}_{\Cal B(\Xi)}$. On the other hand, denoting by
$\parallel\cdot\parallel_{{\gk S}_2(\mathcal L)}$ the Hilbert-Schmidt norm, one has
$$
\begin{aligned}
   \normrum{L_a(b)}{L^2(\Xi)} & =\normrum{\opn{Op}(a\#b)}{{\gk S}_2(\mathcal L)}
\\
&  \le \normrum{\opn{Op}(a)}{\mathbb B(\mathcal L)}
 \,\normrum{\opn{Op}(b)}{{\gk S}_2(\mathcal L)}\\
& =\normrum{\opn{Op}(a)}{\mathbb B(\mathcal
L)}\,\normrum{b}{L^2(\Xi)},
\end{aligned}
$$
hence  $\norm{\opn{Op}(a)}_{\mathbb B(\mathcal L)}\ge \norm{a}_{\Cal B(\Xi)}$. It follows that the norm of the operator
$L_a$ is in fact equal to the norm of $\opn{Op}(a)$ in $\mathbb B(\mathcal L)$. The Rieffel algebra $\mathscr B(\Xi)$ is the
closure of $BC^\infty(\Xi)$ in the norm $a\to \norm{a}_{\mathscr B(\Xi)}=\norm{L_a}_{\mathbb B[L^2(\Xi)]}$, hence
it is isomorphic to $\mathbb B_{{\rm
u}}(\mathcal L)$, the closure of $\mathbb B_{\rm u}^\infty(\mathcal L)=\opn{Op}[BC^\infty(\Xi)]$, as stated.

Now the last statement of the Lemma is trivial if we recall that $\opn{Op}[\S(\Xi)]\subset\mathbb K(\mathcal L)$.
\end{proof}

\section{The canonical mappings}\label{glmoma}

\begin{Definition}\label{glomod}
On $\S\left(\Xi;\AAi\right)$ we introduce {\rm the canonical mappings}
\begin{equation}\label{somorfismv}
[M(F)](X):=\int_\Xi dY\,e^{-i[\![X,Y]\!]}\,\Th_Y\[F\(Y\)\]
\end{equation}
and
\begin{equation}\label{adjuv}
\[M^{-1}(G)\](X):=\int_\Xi dY\,e^{-i[\![X,Y]\!]}\,\Th_{-X}\[G(Y)\]\,.
\end{equation}
\end{Definition}

To give a precise meaning to these relations, use {\it the (symplectic) partial Fourier transform}
\begin{equation*}
\mathfrak F\equiv\mathcal F\otimes
1:\S(\Xi;\AAi)\rightarrow\S(\Xi;\Ai),\ \ \ \ \ (\mathfrak F
F)(X):=\int_\Xi dY e^{-i[\![X,Y]\!]}F(Y)\,.
\end{equation*}
Defining also $C$ by $\[C(F)\](X):=\Th_{X}\[F(X)\]$, we have $M=\mathfrak F\circ C$ and $M^{-1}=C^{-1}\circ\mathfrak F$.

\begin{Proposition}\label{bezu}
The mapping
$M:\left(\S\left(\Xi;\AAi\right),\square\,,\,^\square\,\right)\rightarrow\left(\S\left(\Xi;\Ai\right), \diamond\,,\,^\diamond\,\right)$
is an isomorphism of Fr\'echet $^*$-algebras and $M^{-1}$ is its inverse.
\end{Proposition}

\begin{proof}
The partial Fourier transform is an isomorphism. One also checks that $C$ is an isomorphism of $\S\left(\Xi;\AAi\right)$; this
follows from the explicit form of the seminorms on $\mathcal S(\Xi;\,\AA^\infty)$, from the fact that
$\Th_X$ is isometric and from the formula
$$
\partial^\beta\[\Th_X (F(X))\]=\underset{\gamma\leq\beta}{\sum} C_{\beta\gamma}\Th_X\{\delta^\gamma
\[(\partial^{\beta-\gamma}F)(X)\]\}\,.
$$
With this remarks we conclude that $M=\mathfrak F\circ C$ and $M^{-1}=C^{-1}\circ\mathfrak F$ are reciprocal topological linear
isomorphisms.

We still need to show that $M$ is a $^*$-morphism. {\it For the involution}:
$$
\begin{aligned}
\[M(F)\]^\diamond(X) & =\left\{\int_\Xi dY\,e^{i[\![X,Y]\!]}\Th_Y\[F\(Y\)\]\right\}^* \\
& =\int_\Xi dY\,e^{-i[\![X,Y]\!]}\Th_Y\[F\(Y\)^*\,\] \\
& =\left[M\left(F^\square\right)\right](X)\,.
\end{aligned}
$$
{\it For the product}: it is enough to show that $M^{-1}\[M(F)\diamond M(G)\]=F\square\,G$ for all $F,G\in\S(\Xi;\AAi)$.
One has (iterated integrals):
$$
\begin{aligned}
&(M^{-1}\[MF\diamond MG\])(X)=\int_\Xi dY_1\,e^{-i[\![X,Y_1]\!]}\,
\Th_{-X}\left\{\[MF\diamond MG\]\(Y_1\)\right\} \\
& =\!\int_\Xi\!dY_1\,e^{-i[\![X,Y_1]\!]}\,\Th_{-X}\!\left\{\int_\Xi\!
dY_2\,e^{-\frac{i}{2}\![\![Y_1,Y_2]\!]}\Th_{(Y_2-Y_1)/2}\[(MF)(Y_2)\]
\Th_{Y_2/2}\!\[(MG)(Y_1-Y_2)\]\right\} \\
& =\!\int_\Xi\!dY_1\!\int_\Xi\!dY_2\,e^{-i[\![X,Y_1]\!]}\,e^{-\frac{i}{2}[\![Y_1,Y_2]\!]}\,\Th_{-X}\!\left\{
\Th_{(Y_2-Y_1)/2}\[(MF)(Y_2)\]\Th_{Y_2/2}\[(MG)(Y_1-Y_2)\]\right\} \\
& =\!\int_\Xi\!dY_1\int_\Xi dY_2\,e^{-i[\![X,Y_1]\!]}\,e^{-\frac{i}{2}[\![Y_1,Y_2]\!]}\,\cdot\Th_{(Y_2-Y_1)/2-X}\left\{\int_\Xi
dY_3\,e^{-i[\![Y_2,Y_3]\!]}\,\Th_{Y_3}\[F\(Y_3\)\]\right\} \\
& \quad\quad\quad\ \cdot
\Th_{Y_2/2-X}\left\{\int_\Xi dY_4\,e^{-i[\![Y_1-Y_2,Y_4]\!]}\,\Th_{Y_4}\[G\(Y_4\)\]\right\} \\
& =\int_\Xi\!dY_1\int_\Xi\!dY_2\int_\Xi\!dY_3\int_\Xi\!
dY_4\,e^{-i[\![X,Y_1]\!]}\,e^{-\frac{i}{2}[\![Y_1,Y_2]\!]}
\,e^{-i[\![Y_2,Y_3]\!]}e^{-i[\![Y_1-Y_2,Y_4]\!]}\,\cdot \\
& \quad\quad\quad\ \cdot\Th_{Y_3+(Y_2-Y_1)/2-X}\[F\(Y_3\)\]\Th_{Y_4+Y_2/2-X}\[G\(Y_4\]\) \\
& =2^{4n}\!\int_\Xi\!dY\!\int_\Xi\!dZ\!\int_\Xi\!dY_3\int_\Xi
dY_4\,e^{-2i[\![X,Y_3-Y_4]\!]}\,e^{2i[\![Y,Z]\!]}\,e^{-2i[\![Y_3,Y_4]\!]}\,
\Th_{Y}\[F\(Y_3\)\]\Th_{Z}\[G\(Y_4\)\]\,.
\end{aligned}
$$
For the last equality we made the substitution $\,Y=Y_3+\frac{1}{2}(Y_2-Y_1)-X,\ Z=Y_4+\frac{1}{2}Y_2-X\,$.
Finally, setting $\ Y_3=X-A,\ Y_4=X-B\,$, we get
$$
\(M^{-1}\[MF\diamond MG\]\)(X)=[F\square G](X)=
$$
$$
=2^{4n}\int_\Xi dY\int_\Xi dZ\int_\Xi dA\int_\Xi dB\,e^{-2i[\![A,B]\!]}\,e^{2i[\![Y,Z]\!]}\,
\Th_{Y}\[F\(X-A\)\]\Th_{Z}\[G\(X-B\)\]\,.
$$
\end{proof}

\begin{Remark}\label{alta}
{\rm Let us make some comments about how one could modify the definitions above. We are going to need the notation
$\left[C_\alpha(F)\right](X):=\Theta_{\alpha X}[F(X)]$, where $X\in\Xi,\,F\in\S(\Xi;\Ai)$ (or $F\in L^1(\Xi;\A)$) and $\alpha$
is a real number. All these operations are isomorphisms and our previous transformation $C$ coincides with $C_1$.
The traditional composition law in the twisted crossed product is not (\ref{ucu}), but
\begin{equation*}\label{oco}
(G_1\diamond' G_2)(X):=\int_\Xi
dY\,\kappa(X,Y)\,G_1(Y)\,\Th_{Y}\left[G_2(X-Y)\right].
\end{equation*}
The distinction is mainly an ordering matter and it corresponds to the distinction between the Weyl and the Kohn-Nirenberg forms of
pseudodifferential theory. Applying $C_{1/2}$ leads to an isomorphism between the two algebraic structures. So, if we want
to land in this second realization, we should replace $M=\mathfrak F\,C_1$ with $M':=C_{1/2}\,\mathfrak F\,C_1$, leading explicitly to
$$
[M'(F)](X):=\int_\Xi dY\,e^{-i[\![X,Y]\!]}\,\Th_{Y+X/2}\[F\(Y\)\]\,.
$$
}
\end{Remark}

\section{The $C^*$-algebraic isomorphism}\label{indoc}

We recall that $\Cal K(\Xi)$, with multiplication $\sharp$, has been defined as the R-deformation of the commutative $C^*$-algebra
$C_0(\Xi)$ on which $\Xi$ acts by translations. Then $\Cal K(\Xi)$ is an elementary (hence nuclear) $C^*$-subalgebra of $\Cal
B(\Xi)$, and $\Cal S(\Xi)$ is dense in $\Cal K(\Xi)$. The Fr\'echet $^*$-algebra
$\S(\Xi;\AAi)\equiv\S(\Xi)\hat\otimes\,\AAi$ with the composition law $\square$ given in (\ref{palma}) is dense in the $C^*$-algebra
$\mathscr K(\Xi)\otimes\AA$, that can be viewed (see Remark \ref{uhu} and \cite[Prop.2.1]{Rie2}) as the R-deformation of
$C_0(\Xi)\otimes\A$ with respect to the action of $\Xi\times\Xi$ composed of translations in the first variable and the initial
action $\Th$ in the second.

This section is mainly dedicated to the proof of the next result:

\begin{Theorem}\label{stric}
The mapping $M$ extends to a $C^*$-isomorphism $:\Cal
K(\Xi)\otimes\AA\rightarrow \A\rtimes_\Th^\kappa\Xi$.
\end{Theorem}

The following definition (see \cite[Def.1.2]{Schw} and the concept of {\it differential seminorm} in
\cite[Def.3.1]{BC91}) isolates a situation in which any injective
morphism between a dense $^*$-subalgebra of a $C^*$-algebra and another $C^*$-algebra can be extended to a $C^*$-algebraic monomorphism.

\begin{Definition}\label{B-C}
Let $\mathscr F$ be a dense Fr\'echet subalgebra of a $C^*$-algebra ${\mathscr C}$. We say that \emph{${\mathscr F}$
satisfies the Blackadar-Cuntz condition in ${\mathscr C}$} if the topology on ${\mathscr F}$ is given by a family of seminorms
$\{p_k\}_{k\ge 0}$ such that $p_0$ is the $C^*$-norm giving the topology on ${\mathscr C}$ and
$$
p_k(ab) \le \sum\limits_{i+j=k} p_i(a) p_j(b), \quad a, b\in
{\mathscr F}.
$$
\end{Definition}

Tracing back through \cite{BC91}, one realizes that if $\mathscr F$ satisfies the Blackadar-Cuntz condition in $\mathscr
C$, it is a smooth algebra in the sense of \cite[Def.6.6]{BC91}.
Actually the more general concept of {\it derived seminorm} \cite[Def.5.1]{BC91}
involved in the definition of a smooth algebra is meant to model quotients of differential seminorms. Therefore the following result is in
fact a particular case of \cite[Prop.6.8]{BC91}:

\begin{Proposition}\label{isometry}
Assume $\mathscr F$ is a dense Fr\'echet subalgebra of a $C^*$-algebra ${\mathscr C}$ and satisfies the Blackadar - Cuntz
condition in $\mathscr C$. Then if ${\mathscr D}$ is another $C^*$-algebra and $\Phi\colon {\mathscr F} \mapsto {\mathscr D}$
is an injective $^*$-morphism, then $\Phi$ is isometric for the $C^*$-norm on ${\mathscr C}$.
\end{Proposition}

We now prove Theorem \ref{stric}.

\begin{proof}
The algebra $\S(\Xi)\hat{\otimes}\,\AA^\infty$ is a dense subalgebra of $\Cal K(\Xi) \otimes\,\AA$.
As mentioned before, it can be identified to
$\S\left(\Xi;\AAi\right)$. Proposition \ref{bezu} gives an injective $^*$-morphism
$M:\S\left(\Xi;\AAi\right)\rightarrow\A\rtimes_\Th^\kappa\Xi$ with dense range. If one proves that
$\S(\Xi)\hat{\otimes}\,\AA^\infty$ satisfies Blackadar-Cuntz condition in $\Cal K(\Xi)\otimes \AA$,
Proposition \ref{isometry} shows that $M$ is isometric
for the $C^*$-norm on $\Cal K(\Xi) \otimes \AA$, so it extends to an isomorphism
$:\Cal K(\Xi)\otimes\AA\rightarrow \A\rtimes_\Th^\kappa\Xi$.

To show the Blackadar-Cuntz condition for $\S(\Xi)\hat{\otimes}\,\AA^\infty$, we are going to express it as the space of
smooth vectors for a continuous group action in $\Cal K(\Xi) \otimes \AA$.

Using the Schr\" odinger representation (\ref{sch}) of the Heisenberg group $\mathbb H_\X$ in $\mathcal L=L^2(\X)$, we consider the strongly
continuous representation (by Banach space isomorphisms)
$$
\begin{aligned}
 & \Delta\colon {\bb H}_{\Cal X}\times {\bb H}_{\Cal X} \to \mathbb B[{\bb K}(\mathcal L)]\,,\\
& \Delta({\bar X}, {\bar Y})T = \Pi({\bar
X}) T\Pi(-{\bar Y}), \quad {\bar X}, {\bar Y}\in {\bb H}_{\Cal X}.
\end{aligned}
$$
Notice that $\mathbb K(\mathcal L)$ is an admissible ideal in $\mathbb B(\mathcal L)$, as in \cite[Def.3.8]{BB11}. Also recall
that the Weyl-Pedersen calculus for general
nilpotent Lie groups $G$, introduced in \cite{Pe} and developed in \cite{BB1,BB11,BB4}, particularizes to the usual
Weyl calculus if $G=\mathbb H_\X$ is the Heisenberg group; therefore one can use the results of these articles.
It follows from \cite[Th.4.6 and Th.3.13]{BB11}  (see also \cite{BB4}
and \cite[Th.4.1.4]{Pe}) that the space $\mathbb K(\mathcal L)^\infty$ of smooth vectors of the representation $\Delta$
is precisely $\opn{Op}[\S(\Xi)]$ and that $\opn{Op}:\S(\Xi)\rightarrow \mathbb K(\mathcal L)^\infty$ is a
topological isomorphism of Fr\'echet spaces (a restriction of the isomorphism given by Lemma \ref{precisely}). Hence
$\opn{Op}\hat\otimes\,{\rm id}:\S(\Xi)\hat{\otimes}\,\AA^\infty\rightarrow\mathbb K(\mathcal L)^\infty\hat{\otimes}\,\AA^\infty$
is also an isomorphism of Fr\'echet spaces.
Thus, to finish the proof, it will be enough to show that $\mathbb K(\mathcal L)^\infty\hat{\otimes}\,\AA^\infty$
satisfies the Blackadar-Cuntz condition in $\mathbb K(\mathcal L)\otimes\,\AA$.

We set
$$
\begin{aligned}
 & \Omega\colon {\bb H}_{\Cal X}\times {\bb H}_{\Cal X}\times\Xi \to \mathbb B[{\bb K}(\mathcal L)\otimes\AA]\,,\\
& \Omega({\bar X}, {\bar Y},Z) = \Delta(\bar X,\bar Y)\otimes\Th_Z, \quad {\bar X}, {\bar Y}\in {\bb H}_{\Cal X},\,Z\in\Xi.
\end{aligned}
$$
It is easy to check that $\Omega$ is a strongly continuous representation and that the its space of smooth vectors
$[\mathbb K(\mathcal L)\otimes\,\AA]^\infty$ coincides with the (unique) topological
tensor product $\mathbb K(\mathcal L)^\infty\hat\otimes\,\AAi$.
The transformations $\Omega({\bar X}, {\bar Y},Z)$ are isometric.

It follows that the topology of the tensor product $\mathbb K(\mathcal L)^\infty\hat{\otimes}\,\AAi$ is
also given by the countable family of seminorms
$$
p_k(\Phi):= \sum\limits_{|\alpha|+|\beta|\le k} \frac{1}{\alpha!\beta!}\left\Vert\partial_{\bar X}^{\alpha_1} \partial_{\bar
Y}^{\alpha_2} \partial_Z^\beta\left[\Omega({\bar X}, {\bar Y},Z)\Phi\right]_{\bar X=\bar Y=Z=0}\right\Vert_{\mathbb
K(\mathcal L)\otimes\AA}\,,
$$
where $\alpha=(\alpha_1, \alpha_2)\in\N^{4n+2}$ and $\beta\in\N^{2n}$. When computing on products $\Phi\circ\Psi$, one has to face
the fact that the action $\Delta$ is not automorphic on $\mathbb K(\mathcal L)$.
Note however that when $S$, $T\in {\bb K}(\mathcal L)$ we have
$$
\Delta(\bar X, \bar Y)(S T) = \left[\Delta(\bar X, 0)S\right] \left[\Delta(0, \bar Y)T\right]\,,
$$
implying for all $\Phi,\Psi\in\mathbb K(\mathcal L)\otimes\,\AA$ and all
$(\bar X,\bar Y,Z)\in{\bb H}_{\Cal X}\times{\bb H}_{\Cal X}\times\Xi$
$$
\Omega(\bar X, \bar Y,Z)(\Phi\circ\Psi) = \left[\Omega(\bar X, 0,Z)\Phi\right]\circ \left[\Omega(0, \bar Y,Z)\Psi\right]\,.
$$
Then a simple calculation shows that
$$
p_k(\Phi\circ\Psi)  \le \sum\limits_{i+j=k} p_i(\Phi) p_j(\Psi)
$$
for all $\Phi,\Psi\in \mathbb K(\mathcal L)^\infty\hat\otimes\,\AAi$.
One also has $p_0(\Phi)=\parallel \Phi\parallel_{\mathbb K(\mathcal L)\otimes\AA}\,$, so the proof is finished.
\end{proof}

\begin{Remark}\label{polonezu}
{\rm Let us consider the continuous action $\beta:\Xi\rightarrow{\rm Aut}\left(\A\rtimes_\Th^\kappa\Xi\right)$ given for $G\in
L^1(\Xi;\A)$ by
$$
\left[\beta_Z(G)\right](X):=e^{i[\![X,Z]\!]}G(X)\,,\quad X,Z\in\Xi
$$
(this is the dual action in disguise). Then a short computation gives for any $Z\in\Xi$
\begin{equation}\label{conct}
M\circ\left(\mathcal T_{-Z}\otimes\Th_Z\right)=\beta_Z\circ M\,,
\end{equation}
so actually $M$ can be seen as an isomorphism of $C^*$-dynamical systems.
Thus the twisted crossed product $\A\rtimes_\Th^\kappa\Xi$ endowed with the action $\beta$ can be seen as (an isomorphic copy of)
the Rieffel deformation of the $C^*$-algebra $C_0(\Xi;\A)$.}
\end{Remark}

\section{Applications}\label{indoiuc}

One can rephrase Theorem \ref{stric} by saying that $\A\rtimes_\Th^\kappa \Xi$ is (isomorphic to) the stable algebra
of $\AA$. In particular, $\AA$ and $\A\rtimes_\Th^\kappa \Xi$ are stably isomorphic. Therefore they have identical representation
theories (indexed by covariant representations of the system $(\A,\Th,\kappa,\Xi)$), isomorphic ideal lattices and there
are canonical homeomorphisms between the corresponding spaces of primitive ideals \cite{RW}.

\medskip
We investigate now the interplay between the canonical maps and $\Xi$-morphisms. Let $\left(\A^j,\Theta^j,\Xi,\kappa\right)$,
$j=1,2$, be two sets of classical data and let $\R:\A^1\rightarrow\A^2$ a $\Xi$-morphism, \ie\  a $C^*$-morphism intertwining the two
actions $\Theta^1,\Theta^2$. Then $\R$ acts coherently on $C^\infty$-vectors ($\,\R[\A^{1,\infty}]\subset\A^{2,\infty}\,$)
and extends to a morphism $\RR:\AA^1\rightarrow\AA^2$ of the R-quantized $C^*$-algebras that also intertwines the corresponding
actions (see \cite{Rie1}). On the other hand \cite{PR1,PR2} another $C^*$-morphism
$\R^\rtimes:\A^1\rtimes_{\Th^1}^\kappa\Xi\rightarrow\A^2\rtimes_{\Th^2}^\kappa\Xi$ is assigned canonically to $\R$, uniquely defined by
$$
\[\R^\rtimes(F)\](X):=\R\[F(X)\],\ \ \ \ \ \forall F\in L^1(\Xi;\A^1)\,.
$$

\begin{Proposition}\label{ulema}
Denoting by ${\rm id}$ the identical map on $\mathscr K(\Xi)$ and by $M^j$ the canonical map for the data $(\A^j,\Th^j,\kappa,\Xi)$,
one has
\begin{equation}\label{rarucu}
\mathcal R^\rtimes\circ M^1=M^2\circ({\rm id}\otimes\mathfrak R)\,.
\end{equation}
\end{Proposition}

\begin{proof}
It is enough to compute on $F\in\S(\Xi;\AAi)$:
$$
\begin{aligned}
\left[(\mathcal R^\rtimes\circ M^1)(F)\right](X) & =\mathcal R[(M^1F)(X)] \\
& =\mathcal R\left[\int_\Xi dY\,e^{-i[\![X,Y]\!]}\Th^1_Y(F(Y))\right] \\
& =\int_\Xi dY\,e^{-i[\![X,Y]\!]}\mathcal R\left[\Th^1_Y(F(Y))\right] \\
& =\int_\Xi dY\,e^{-i[\![X,Y]\!]}\Th^2_Y[\mathcal R(F(Y))]  \\
& =M^2[({\rm id}\otimes\mathfrak R)F](X)\,.
\end{aligned}
$$
\end{proof}

The next two results have been proved in \cite{Rie3} without asking the skew-symmetric operator $J$ to be non-degenerate
(see also \cite{Kp}). The proofs relying on Theorem \ref{stric} are very simple:

\begin{Corollary}\label{kteo}
The $C^*$-algebras $\A$ and $\AA$ have the same $K$-groups.
\end{Corollary}

\begin{proof}
Since $\AA$ and $\A\rtimes_\Th^\kappa \Xi$ are stably isomorphic, they have the same $K$-theory \cite{Bl}. On the other hand, by
'the stabilization trick' \cite{PR1}, $\A\rtimes_\Th^\kappa \Xi$ is stably isomorphic to a usual (untwisted) crossed product
$(\A\otimes\mathscr K)\rtimes_\Gamma \Xi$ associated to an action $\Gamma$ of $\Xi$ on the tensor product of $\A$ with an elementary
algebra $\Cal K$. The vector space $\Xi$ has even dimension, hence
by Connes' Thom isomorphism \cite{Bl} the $K$-groups of the crossed product coincides with the $K$-groups of
$\A\otimes\mathscr K$, i.e. with those of $\A$.
\end{proof}

\begin{Corollary}\label{kteo}
The $C^*$-algebras $\A$ and $\AA$ are simultaneously nuclear.
\end{Corollary}

\begin{proof}
The argument is analogue to the previous one. One must also recall \cite[Th.15.8.2]{Bl} that nuclearity is preserved by stable isomorphism
and that the crossed product with a commutative group of a $C^*$-algebra $\mathcal B$ is nuclear iff $\mathcal B$ is nuclear (for the converse use Takai duality).
\end{proof}

If $\A$ is commutative, by Gelfand theory, it is isomorphic (and will be identified) to $C_0(\Si)$, the $C^*$-algebra of all complex
continuous functions on the locally compact space $\Si$ which converge to zero at infinity. The space $\Si$ is a homeomorphic
copy of the Gelfand spectrum of $\A$ and it is compact iff $\A$ is unital. Then the group $\Th$ of automorphisms is induced by an
action (also called $\Th$) of $\Xi$ by homeomorphisms of $\Si$. We are going to use the convention
$$
\left[\Th_X(f)\right](\si):=f\left[\Th_X(\si)\right],\ \ \ \
\forall\,\si\in\Si,\,X\in\Xi,\,f\in\A\,,
$$
as well as the notation $\Th_X(\si)=\Th(X,\si)$ for the $X$-transform of the point $\si$\,.
Let us set $\AA=:\mathfrak C(\Si)$ for the
(non-commutative) Rieffel $C^*$-algebra associated to $C_0(\Si)$ by deformation and
$\mathfrak C^\infty(\Si)$ for the space of smooth vectors under the action $\Th$.

\begin{Remark}\label{ciudat}
{\rm A rather surprising picture follows from a symmetry argument.
Using \cite[Th.6.5]{Rie1}, one gets easily an isomorphism $\Cal K(\Xi)\otimes\A\cong\AA\rtimes_\Th^{\overline\kappa}\Xi$.
The complex conjugated cocycle $\overline\kappa$ is defined by the symplectic form $-[\![\cdot,\cdot]\!]$.
One usually thinks of $\A$ as a rather simple $C^*$-algebra, giving after deformation a more complicated one $\AA$.
In the commutative case, for instance, we might be surprised that the twisted crossed product $\CC(\Si)\rtimes_\Th^{\overline\kappa}\Xi$
decomposes as $C_0(\Si;\Cal K(\Xi))$.}
\end{Remark}

\begin{Remark}\label{primitive}
{\rm Twisted crossed products with commutative $C^*$-algebras are discussed in \cite{Pa}. In some situation their primitive ideal space is understood
(as a topological space) and this can be transferred by our stable isomorphism to the level of $\CC(\Si)$.
By \cite[Ex.4.3]{Pa} for instance, if the action $\Th$ is free (all the isotropy groups are trivial), then ${\rm Prim}[\CC(\Si)]$
is homeomorphic to the quasiorbit space $Q^\Th(\Si)$. If in addition $\Th$ is minimal, $\CC(\Si)$ will be a simple $C^*$-algebra.
If $\Th$ is minimal without being free, the situation is described in \cite[Ex.4.11]{Pa}.}
\end{Remark}

\begin{Remark}\label{contrace}
{\rm Assume that $\Xi$ act freely on $\Si$. By Theorem \ref{stric} and \cite[Th.4.5]{Pa}, $\CC(\Si)$ is a continuous trace $C^*$-algebra
if and only if the action $\Th$ is proper.}
\end{Remark}

We discuss shortly orthogonality matters. On $\Si$ we pick a $\Th$-invariant measure $d\si$.
The relationship between the spaces $\S(\Xi;\A^\infty)$ and $L^2(\Xi\times\Si)$ depends on the assumptions we impose on
$(\Si,d\si)$. If $d\si$ is a finite measure, for instance, one has $\S(\Xi;\A^\infty)\subset L^2(\Xi\times\Si)$. Anyhow, the
canonical map can be defined independently on $L^2(\Xi\times\Si)$.

\begin{Proposition}\label{one}
One has {\rm the orthogonality relations} valid for $F,G\in L^2(\Si\times\Xi)$\,:
\begin{equation}\label{tion}
\<\overline{M(F)},M(G)\>_{\Xi\times\Si}=\<\overline{F},G\>_{\Xi\times\Si}\,.
\end{equation}
Thus the operator $M:L^2(\Xi\times\Si)\rightarrow L^2(\Xi\times\Si)$ is unitary.
\end{Proposition}

\begin{proof}
It is enough to note that $M=\mathfrak F\circ C$ and to use the fact that $\mathfrak F$ and $C$ are isomorphisms of
$L^2(\Si\times\Xi)$ if $d\si$ is $\Th$-invariant.
\end{proof}

{\bf Acknowledgements:}

I. Belti\c t\u a is  supported by {\textit Proyecto Fondecyt No. 1085162.}

M. M\u antoiu is  supported by {\textit Proyecto Fondecyt No. 1085162} and by {\textit N\'ucleo Cientifico ICM P07-027-F
"Mathematical Theory of Quantum and Classical Magnetic Systems".}

\end{document}